\documentclass[11pt,a4paper,reqno]{amsart}

\usepackage[T1]{fontenc}
\usepackage{lmodern}
\usepackage[margin=25mm]{geometry}
\usepackage{mathtools,amssymb,booktabs,microtype,needspace}
\usepackage[hidelinks]{hyperref}

\newtheorem{theorem}{Theorem}[section]
\newtheorem{lemma}[theorem]{Lemma}
\newtheorem{proposition}[theorem]{Proposition}

\theoremstyle{remark}

\numberwithin{equation}{section}

\newcommand{\M}{\mathbf M}
\newcommand{\V}{\mathsf V}
\newcommand{\p}{\mathbf p}
\newcommand{\sseven}{S_7}
\newcommand{\J}{J}
\newcommand{\K}{K}
\newcommand{\Om}{\Omega}
\newcommand{\comp}[1]{\overline{#1}}
\allowdisplaybreaks[1]

\hypersetup{
  pdftitle={A separating identity for matrix semirings over S7},
  pdfauthor={Lili Wang, Jiewei Zhang, and Aifa Wang},
  pdfsubject={Strict inclusion between the varieties generated by M2(S7) and M3(S7)},
  pdfkeywords={additively idempotent semiring, matrix semiring, variety, identity}
}

\title[A separating identity for matrix semirings]
{A separating identity for matrix semirings over $S_7$}
\author{Lili Wang}
\address{School of Mathematical Sciences, Chongqing University of Technology,
Chongqing 400054, P.R. China}
\email{wllaf@cqut.edu.cn}
\author{Jiewei Zhang}
\address{School of Mathematical Sciences, Chongqing University of Technology,
Chongqing 400054, P.R. China}
\email{1114529127@qq.com}
\author{Aifa Wang\textsuperscript{*}}
\thanks{\textsuperscript{*}Corresponding author: Aifa Wang
(\href{mailto:wangaf@cqut.edu.cn}{wangaf@cqut.edu.cn}).}
\address{School of Mathematical Sciences, Chongqing University of Technology,
Chongqing 400054, P.R. China}
\email{wangaf@cqut.edu.cn}
\date{}

\begin{document}

\begin{abstract}
Jiao and Ren proved that the varieties generated by the matrix semirings
$\M_n(S_7)$ form an ascending chain of nonfinitely based varieties,
and asked whether consecutive members are equal for all $n\geq2$.
We give a negative answer by exhibiting an identity in five variables
that holds in $\M_2(S_7)$ and fails in $\M_3(S_7)$.
Consequently, $\V(\M_2(S_7))$ is a proper subvariety of
$\V(\M_n(S_7))$ for every $n\geq3$.
The proof uses the cubic-power property established by Jiao and Ren
and a description of products equal to the constant matrix $[a]_n$.
The identity forces two distinct indices to exhaust the index set
in dimension two; an additional index yields an explicit counterexample
in dimension three.
\end{abstract}

\maketitle

\section{Introduction and preliminaries}

An \emph{additively idempotent semiring}, or \emph{ai-semiring}, is an
algebra $(S,+,\cdot)$ such that $(S,+)$ is a commutative idempotent
semigroup, $(S,\cdot)$ is a semigroup, and multiplication distributes
over addition on both sides. All algebras in this paper are considered
in the signature $(+,\cdot)$, without nullary operations. For an
ai-semiring $S$, let $\V(S)$ denote the variety generated by $S$.
By Birkhoff's theorem~\cite{Bir35}, two ai-semirings generate the same
variety if and only if they satisfy the same identities.

The finite basis problem asks whether the identities of an algebra
follow from finitely many identities. Zhao et al.~\cite{ZRCS20}
proved that all three-element ai-semirings, with the possible exception
of $S_7$, are finitely based. Jackson, Ren and Zhao~\cite{JRZ22}
proved that $S_7$ is nonfinitely based, completing this classification.
Gao et al.~\cite{GJRZ25} further studied finite basis properties and
the subvariety structure associated with $S_7$.

For a positive integer $n$, let $\M_n(S)$ be the ai-semiring of all
$n\times n$ matrices over $S$, with entrywise addition and the usual
matrix multiplication. Dolinka~\cite{Dol09} established inherent
nonfinite basability results for ai-semirings, including semirings of
binary relations on finite sets. Jiao and Ren~\cite{JRtwo26} classified
the finite basis problem for matrix semirings over two-element
ai-semirings.

In their subsequent study of $\M_n(S_7)$, Jiao and Ren~\cite{JR26}
proved that, for every
ai-semiring $S$ and every $n\geq2$, the semiring $\M_n(S)$ embeds into
$\M_{n+1}(S)$. Hence the varieties $\V(\M_n(S))$ form an ascending
chain. They also proved that every $\M_n(\sseven)$ is nonfinitely
based and obtained
\begin{equation}\label{eq:chain}
\V(\sseven)<\V(\M_2(\sseven))\leq\V(\M_3(\sseven))\leq\cdots.
\end{equation}
Here $\sseven$ is the three-element ai-semiring whose operations are
given in Table~\ref{tab:s7}.

\begin{table}[ht]
\centering
\caption{The Cayley tables of $\sseven$.}\label{tab:s7}
\renewcommand{\arraystretch}{1.12}
\begin{tabular}{c|ccc}
$+$ & $1$ & $a$ & $\infty$\\\hline
$1$ & $1$ & $\infty$ & $\infty$\\
$a$ & $\infty$ & $a$ & $\infty$\\
$\infty$ & $\infty$ & $\infty$ & $\infty$
\end{tabular}
\qquad\qquad
\begin{tabular}{c|ccc}
$\cdot$ & $1$ & $a$ & $\infty$\\\hline
$1$ & $1$ & $a$ & $\infty$\\
$a$ & $a$ & $\infty$ & $\infty$\\
$\infty$ & $\infty$ & $\infty$ & $\infty$
\end{tabular}
\end{table}

In Section~6 of the second arXiv version of~\cite{JR26}, dated
17 August 2026, Jiao and Ren asked whether
\begin{equation}\label{eq:question}
\V(\M_n(\sseven))=\V(\M_{n+1}(\sseven))
\qquad\text{for every }n\geq2.
\end{equation}
Their Proposition~5.8 shows that the multiplicative reduct of
$\M_n(\sseven)\setminus\{[1]_n\}$ is $5$-nilpotent for every $n\geq2$,
which led them to suggest that~\eqref{eq:question} might hold.
We show that the inclusion from dimension two to dimension three is
strict.

Let $e,b,x,y,z$ be distinct variables and put
\begin{equation}\label{eq:p}
\p=e^3be^3+exye+bx+ye+z^2.
\end{equation}
Our main result is the following.

\begin{theorem}\label{thm:main}
For every integer $n\geq2$, the identity
\begin{equation}\label{eq:identity}
\p\approx\p+bxe
\end{equation}
holds in $\M_n(\sseven)$ if and only if $n=2$.
Consequently,
\[
\V(\M_2(\sseven))<\V(\M_n(\sseven))
\qquad\text{for every }n\geq3.
\]
\end{theorem}

The symbol $\p$ in~\eqref{eq:identity} abbreviates the term
in~\eqref{eq:p}. In particular, this is an identity in five variables
in the original signature; none of its symbols denotes a prescribed
matrix. The result gives a negative answer to~\eqref{eq:question}.
It does not determine whether the chain stabilizes at a later term.

For $c\in\sseven$, write $[c]_n$ for the constant $n\times n$ matrix
with entry $c$, and set
\[
\J_n=[1]_n,\qquad \K_n=[a]_n,\qquad \Om_n=[\infty]_n.
\]
The matrix $\Om_n$ is the additive maximum and the multiplicative
zero of $\M_n(\sseven)$. The matrix $\J_n$ is idempotent under
multiplication, but is not a multiplicative identity when $n\geq2$.

\section{Elementary properties}

We shall repeatedly use the following consequence of Table~\ref{tab:s7}:
if $s_1,\ldots,s_m\in\sseven$, $m\geq1$, and $c\in\{1,a\}$, then
\begin{equation}\label{eq:flat}
s_1+\cdots+s_m=c
\quad\Longleftrightarrow\quad
s_1=\cdots=s_m=c.
\end{equation}
Also, a product equals $1$ precisely when all its factors are $1$;
a product equals $a$ precisely when exactly one factor is $a$ and
all the others are $1$.

The first lemma records the cubic-power property from
\cite[Corollary~5.7]{JR26}. We include a short proof for completeness.

\begin{lemma}\label{lem:cube}
Let $n\geq2$ and $A\in\M_n(\sseven)$. Then
\begin{equation}\label{eq:cube}
A^3=
\begin{cases}
\J_n,&A=\J_n,\\
\Om_n,&A\ne\J_n.
\end{cases}
\end{equation}
\end{lemma}

\begin{proof}
For $1\leq i,j\leq n$, we have
\[
(A^3)_{ij}=\sum_{k,l=1}^n A_{ik}A_{kl}A_{lj}.
\]
If $(A^3)_{ij}=1$, every summand is $1$ by~\eqref{eq:flat}.
In particular, $A_{kl}=1$ for all $k,l$, and hence $A=\J_n$.

Suppose that $(A^3)_{ij}=a$. Taking $k=l=i$ gives
$A_{ii}^2A_{ij}=a$. Thus $A_{ii}=1$ and $A_{ij}=a$.
However, the summand with $k=j$ and $l=i$ is
$A_{ij}A_{ji}A_{ij}=\infty$, a contradiction.
Therefore $A^3$ has no entry equal to $a$. If $A\ne\J_n$, it also
has no entry equal to $1$, so $A^3=\Om_n$.
The equality $\J_n^3=\J_n$ is immediate.
\end{proof}

For $u\in\{1,a\}$, define its complement by
$\comp1=a$ and $\comp a=1$.

\begin{lemma}\label{lem:factor}
Let $n\geq2$ and $X,Y\in\M_n(\sseven)$.
\begin{enumerate}
\item If $XY=\J_n$, then $X=Y=\J_n$.
\item The equality $XY=\K_n$ holds if and only if there exist
$u_1,\ldots,u_n\in\{1,a\}$ such that
\begin{equation}\label{eq:profile}
X_{rk}=u_k,\qquad Y_{ks}=\comp{u_k}
\qquad(1\leq r,k,s\leq n).
\end{equation}
\end{enumerate}
\end{lemma}

\begin{proof}
If $XY=\J_n$, then~\eqref{eq:flat} gives $X_{rk}Y_{ks}=1$
for all $r,k,s$. Both factors must be $1$, proving the first assertion.

Suppose that $XY=\K_n$. Again by~\eqref{eq:flat},
\begin{equation}\label{eq:pair}
X_{rk}Y_{ks}=a\qquad(1\leq r,k,s\leq n).
\end{equation}
Fix $k$ and set $u_k=X_{1k}$. Equation~\eqref{eq:pair} implies
that $u_k\in\{1,a\}$ and $Y_{ks}=\comp{u_k}$ for every $s$.
Using $s=1$ in~\eqref{eq:pair} then yields $X_{rk}=u_k$ for
every $r$. This proves~\eqref{eq:profile}.
Conversely,~\eqref{eq:profile} makes every summand in each entry
of $XY$ equal to $a$, and hence $XY=\K_n$.
\end{proof}

We shall also use the equality
\begin{equation}\label{eq:sandwich}
\J_nA\J_n=\left[\sum_{r,s=1}^n A_{rs}\right]_n.
\end{equation}
It follows from~\eqref{eq:flat} that, for $c\in\{1,a\}$,
the occurrence of $c$ in any entry of $\J_nA\J_n$ forces $A=[c]_n$.

We also give the embedding argument underlying the ascending chain;
see~\cite[Theorem~2.5]{JR26}.

\begin{lemma}\label{lem:embedding}
Let $S$ be an ai-semiring and let $2\leq m\leq n$. Then
$\M_m(S)$ embeds into $\M_n(S)$.
\end{lemma}

\begin{proof}
Choose a surjection $\sigma:\{1,\ldots,n\}\to\{1,\ldots,m\}$,
and define $\Phi(A)_{ij}=A_{\sigma(i),\sigma(j)}$.
Surjectivity makes $\Phi$ injective, and it preserves addition
entrywise. Moreover, additive idempotence gives
\begin{align*}
(\Phi(A)\Phi(B))_{ij}
&=\sum_{k=1}^n A_{\sigma(i),\sigma(k)}B_{\sigma(k),\sigma(j)}\\
&=\sum_{r=1}^m A_{\sigma(i),r}B_{r,\sigma(j)}
=\Phi(AB)_{ij}.
\end{align*}
Indeed, every index $r$ occurs at least once as $\sigma(k)$,
and repeated summands may be deleted. Thus $\Phi$ is an embedding.
\end{proof}

\section{The separating identity}

We first prove the positive part of Theorem~\ref{thm:main}.

\begin{proposition}\label{prop:two}
The ai-semiring $\M_2(\sseven)$ satisfies~\eqref{eq:identity}.
\end{proposition}

\begin{proof}
Let $E,B,X,Y,Z$ be arbitrary matrices in $\M_2(\sseven)$, and put
\[
F=E^3BE^3+EXYE+BX+YE+Z^2.
\]
We prove $F=F+BXE$ entrywise. Fix $i,j\in\{1,2\}$.
If $F_{ij}=\infty$, then $(F+BXE)_{ij}=F_{ij}$ automatically.
Suppose, therefore, that
\[
F_{ij}=c\in\{1,a\}.
\]
By~\eqref{eq:flat}, each of the five summands defining $F$ has
$(i,j)$-entry equal to $c$.

Since $(E^3BE^3)_{ij}=c$, Lemma~\ref{lem:cube} implies $E=\J_2$:
otherwise $E^3BE^3=\Om_2$. Applying~\eqref{eq:sandwich} to
the first two summands gives
\begin{equation}\label{eq:reduction}
B=[c]_2,\qquad XY=[c]_2.
\end{equation}
We consider the two possible values of $c$.

\smallskip
\noindent\emph{Case 1: $c=1$.}
By~\eqref{eq:reduction} and Lemma~\ref{lem:factor}(1),
$B=X=Y=\J_2$. Hence $BXE=\J_2$, so $(BXE)_{ij}=F_{ij}$.

\smallskip
\noindent\emph{Case 2: $c=a$.}
Now $B=\K_2$ and $XY=\K_2$. By Lemma~\ref{lem:factor}(2),
there exist $u_1,u_2\in\{1,a\}$ such that
\[
X_{rk}=u_k,\qquad Y_{ks}=\comp{u_k}
\qquad(r,k,s\in\{1,2\}).
\]
The third and fourth summands in $F$ yield
\[
a=(BX)_{ij}=a u_j,
\qquad
a=(YE)_{ij}=\comp{u_i}.
\]
Consequently,
\begin{equation}\label{eq:endpoints}
u_i=u_j=1.
\end{equation}

The fifth summand gives $(Z^2)_{ij}=a$. This forces $i\ne j$.
Indeed, if $i=j$, the expansion of $(Z^2)_{ii}$ contains the summand
$Z_{ii}^2\in\{1,\infty\}$, contrary to~\eqref{eq:flat}.
Since the index set has only two elements, $\{i,j\}=\{1,2\}$.
It follows from~\eqref{eq:endpoints} that $u_1=u_2=1$, and hence
$X=\J_2$. Thus $BXE=\K_2$, giving $(BXE)_{ij}=F_{ij}$ in this
case as well.

In every case, $(F+BXE)_{ij}=F_{ij}$. Since $i,j$ were arbitrary,
$F=F+BXE$, as required.
\end{proof}

\Needspace{6\baselineskip}
We next construct a counterexample in every higher dimension.

\begin{proposition}\label{prop:higher}
For every $n\geq3$, the ai-semiring $\M_n(\sseven)$ does not
satisfy~\eqref{eq:identity}.
\end{proposition}

\begin{proof}
Fix $n\geq3$. Set $u_3=a$ and $u_k=1$ for $k\ne3$, and take
$E=\J_n$, $B=\K_n$. Define
\begin{equation}\label{eq:general-witness}
X_{rk}=u_k,\qquad Y_{ks}=\comp{u_k},\qquad
Z_{rs}=\begin{cases}
a,&s=2\text{ and }r\ne2,\\
1,&\text{otherwise}.
\end{cases}
\end{equation}
By Lemma~\ref{lem:factor}, $XY=\K_n$, and hence
\[
E^3BE^3=EXYE=\K_n.
\]
Since $u_1=u_2=1$, we also have
\[
(BX)_{12}=\sum_{k=1}^n a u_2=a,
\qquad
(YE)_{12}=\sum_{k=1}^n\comp{u_1}=a.
\]
For the square term, every summand at coordinate $(1,2)$ equals $a$:
\[
Z_{1k}Z_{k2}=
\begin{cases}
a\cdot1,&k=2,\\
1\cdot a,&k\ne2.
\end{cases}
\]
Thus $(Z^2)_{12}=a$. If $F$ is the value of $\p$ under this
substitution, all five summands have $(1,2)$-entry $a$, so $F_{12}=a$.

On the other hand, $u_3=a$ gives $(BX)_{i3}=a u_3=\infty$
for every $i$. Each entry of $BXE$ therefore contains the summand
$(BX)_{i3}E_{3j}=\infty$, and hence $BXE=\Om_n$.
Consequently, $(F+BXE)_{12}=\infty\ne a=F_{12}$, proving the claim.
\end{proof}

\begin{proof}[Proof of Theorem~\ref{thm:main}]
Propositions~\ref{prop:two} and~\ref{prop:higher} establish the
assertion about~\eqref{eq:identity}. By Lemma~\ref{lem:embedding},
$\M_2(\sseven)$ embeds into $\M_n(\sseven)$ for every $n\geq3$,
so $\V(\M_2(\sseven))\leq\V(\M_n(\sseven))$.
Every member of $\V(\M_2(\sseven))$ satisfies~\eqref{eq:identity},
whereas $\M_n(\sseven)$ does not. The inclusion is therefore strict.
\end{proof}

\section{Concluding remarks}

Theorem~\ref{thm:main} strengthens the beginning of~\eqref{eq:chain} to
\[
\V(\sseven)<\V(\M_2(\sseven))<\V(\M_3(\sseven))
\leq\V(\M_4(\sseven))\leq\cdots.
\]
Thus the common $5$-nilpotency property of the subsemirings
$\M_n(\sseven)\setminus\{\J_n\}$ does not imply equality of the
varieties generated by the full matrix semirings.

The identity records whether two distinct indices exhaust the index
set. At a coordinate $(i,j)$ where $\p$ has value $c\ne\infty$,
the terms $e^3be^3$ and $exye$ force $E=\J_n$ and $B=XY=[c]_n$.
For $c=a$, the terms $bx$ and $ye$ force $u_i=u_j=1$
in~\eqref{eq:profile}, while $z^2$ forces $i\ne j$.
A third index permits a parameter $u_k=a$, which makes $BXE=\Om_n$.

The inclusions between consecutive terms of dimension at least three
are not determined by this argument. In particular, it remains to
decide whether $\V(\M_3(\sseven))=\V(\M_4(\sseven))$ and whether
the chain eventually stabilizes.

\Needspace{20\baselineskip}


\begin{thebibliography}{7}

\bibitem{Bir35}
Birkhoff, G.:
On the structure of abstract algebras.
Proc. Cambridge Philos. Soc. \textbf{31}(4), 433--454 (1935).
\url{https://doi.org/10.1017/S0305004100013463}

\bibitem{Dol09}
Dolinka, I.:
A class of inherently nonfinitely based semirings.
Algebra Universalis \textbf{60}(1), 19--35 (2009).
\url{https://doi.org/10.1007/s00012-008-2084-y}

\bibitem{GJRZ25}
Gao, Z., Jackson, M., Ren, M.M., Zhao, X.Z.:
The finite basis problem for additively idempotent semirings that
relate to $S_7$.
Preprint, \href{https://arxiv.org/abs/2501.19049}{arXiv:2501.19049}
(2025).
\url{https://doi.org/10.48550/arXiv.2501.19049}

\bibitem{JRZ22}
Jackson, M., Ren, M.M., Zhao, X.Z.:
Nonfinitely based ai-semirings with finitely based semigroup reducts.
J. Algebra \textbf{611}, 211--245 (2022).
\url{https://doi.org/10.1016/j.jalgebra.2022.07.042}
\href{https://arxiv.org/abs/2112.13918}{Preprint: arXiv:2112.13918}.

\bibitem{JRtwo26}
Jiao, J., Ren, M.M.:
The finite basis problem for matrix semirings over a two-element
additively idempotent semiring.
Semigroup Forum (2026). Published online 5 August 2026.
\url{https://doi.org/10.1007/s00233-026-10667-x}
\href{https://arxiv.org/abs/2602.06972}{Preprint: arXiv:2602.06972}.

\bibitem{JR26}
Jiao, J., Ren, M.M.:
The finite basis problem for matrix semirings $\M_n(S_7)$.
Preprint, \href{https://arxiv.org/abs/2607.09677v2}{arXiv:2607.09677v2}
(17 August 2026).
\url{https://doi.org/10.48550/arXiv.2607.09677}

\bibitem{ZRCS20}
Zhao, X.Z., Ren, M.M., Crvenkovi\'c, S., Shao, Y., Dapi\'c, P.:
The variety generated by an ai-semiring of order three.
Ural Math. J. \textbf{6}(2), 117--132 (2020).
\url{https://doi.org/10.15826/umj.2020.2.012}

\end{thebibliography}
\end{document}